\documentclass{amsart}

\usepackage{amsthm}
\usepackage{amsrefs}
\usepackage{amsmath}
\usepackage{amssymb}
\usepackage{booktabs}

\newtheorem{theorem}{Theorem}
\newtheorem{lemma}[theorem]{Lemma}
\newtheorem{proposition}[theorem]{Proposition}
\newtheorem{corollary}[theorem]{Corollary}

\theoremstyle{definition}
\newtheorem{definition}[theorem]{Definition}

\newtheorem{remark}[theorem]{Remark}
\theoremstyle{remark}
\newtheorem*{ack}{Acknowledgment}

\numberwithin{equation}{section}

\DeclareMathOperator{\Char}{char}
\DeclareMathOperator{\Mod}{\,mod}
\DeclareMathOperator{\GF}{GF}
\DeclareMathOperator{\GL}{GL}
\DeclareMathOperator{\fS}{\mathfrak{S}}
\DeclareMathOperator{\fSone}{\mathfrak{S}_I}
\DeclareMathOperator{\fStwo}{\mathfrak{S}_{II}}
\DeclareMathOperator{\fSthree}{\mathfrak{S}_{III}}
\DeclareMathOperator{\End}{End}

\newcommand{\Sym}{\tilde{S}_n}
\newcommand{\sym}{\tilde{S}_{n-1}}

\newcommand{\pdiv}{p\mid}
\newcommand{\pnotdiv}{p\nmid}
\newcommand{\arr}{\uparrow\uparrow}

\begin{document}
\title[On a Construction of the Basic Spin Representations]{On a Construction of the Basic Spin Representations of Symmetric Groups}

\author[L.~Maas]{Lukas A.~Maas}
\address{Institut f\"ur Experimentelle Mathematik, 
Universit\"at Duisburg--Essen, Ellernstr. 29, 45326 Essen, Germany}
\curraddr{School of Mathematics,
University of Birmingham,
Birmingham B15 2TT, United Kingdom}
\email{lukas.maas@iem.uni-due.de}

\begin{abstract}
We present an inductive method for constructing the basic spin representations 
of the double covers of the symmetric groups over fields of any characteristic.
\end{abstract}
\maketitle
\section*{Introduction}
Basic spin representations are the smallest faithful representations of the double covers
of the symmetric groups. 
The ordinary basic spin representation $X_n$ was given explicitly by 
I.~Schur \cite{Schur1911}, see \cite{HoffmanHumphreys1992}*{Chapter 6} for a modern treatment. 
Let $p$ be a prime. By a result of D.~B.~Wales \cite{Wales1978}, 
the $p$-modular reduction of $X_n$ is irreducible unless $n$ is odd and divisible by $p$,
in which case there exist exactly two distinct composition factors. We call these $p$-modular 
representations also basic spin. In this paper we describe a construction 
of the basic spin representations over splitting fields of any characteristic, that is, 
at most quadratic extensions of the prime field. The method of 
construction is inductive with respect to $n$.\\ 
Let $\Sym$ be the group generated by elements $z, t_1, \ldots, t_{n-1}$ subject to the relations
\begin{align*}
 	z^2 = &\ 1, \\
      t_i^2 = &\ z \quad (1\le i\le n-1), \\
      (t_it_{i+1})^3 = &\ z \quad (1\le i\le n-2), \\
      (t_it_j)^2 = &\ z\quad (1\le i, j\le n-1,\ |i-j| > 1).
\end{align*}
Provided $n\geq 4$, it is well-known that $\Sym$ is a double cover of the symmetric group 
$S_n$, and the projection of $\Sym$ onto $S_n$ may be defined by sending the generators $t_i$ to 
$(i,i+1)\in S_n$.  
There is another double cover $\hat{S}_n$ of $S_n$ 
which is non-isomorphic to $\Sym$ for $n\neq 6$, and 
$\hat{S}_n$ can be described as generated by elements
$z, s_1, \ldots, s_{n-1}$ subject to the relations
$z^2=1=s_i^2=(zs_i)^2\ (1\le i\le n-1)$, 
$(s_is_{i+1})^3=1\ (1\le i\le n-2)$, and
$s_is_j = zs_js_i\ (1\le i, j\le n-1,\ |i-j| > 1)$.  
If we consider representations over a splitting field of characteristic $p=2$, 
then every representation of $\Sym$ or $\hat{S}_n$ contains the central element $z$ in its kernel, 
and essentially, we deal with $2$-modular representations of $S_n$. 
If $p\neq 2$, then there is a bijective 
correspondence between faithful irreducible representations of $\Sym$ and 
faithful irreducible representations of $\hat{S}_n$. Let $\omega\in K$ be a primitive fourth root 
of unity; then this correspondence can be realized through
$R\mapsto \omega R$ where $R$ is a representation of $\Sym$ and $\omega R$ is a representation 
of $\hat{S}_n$ defined by $z\mapsto -id$ and $s_i\mapsto \omega\cdot t_i^R$ for $i=1,\ldots,n-1$.
Hence we may concentrate below on representations of $\Sym$.

\section*{Basic Spin Representations}
We adopt the following degree notation from \cite{KleshchevTiep2004}: for $n\geq 4$ and
$p\geq 0$ let 
\begin{align*}
\delta(\Sym) =& \begin{cases}
	\ 2^{k-1} & \quad\mbox{if}\ n=2k,\cr
	\ 2^{k-1} & \quad\mbox{if}\ n=2k+1\ \mbox{and}\ \pdiv n,\cr
	\ 2^{k} & \quad\mbox{if}\ n=2k+1\ \mbox{and}\ \pnotdiv n.
\end{cases}
\end{align*}
Let $sgn$ be the representation of 
$\Sym$ which is given by inflation of the sign representation of $S_n$. The \textit{associate 
representation} of an $\Sym$-representation $R$ is $sgn\otimes R$. 
\begin{theorem}\label{bsp} For $n\geq 4$, let $X_n=X_n^+$ be the basic (spin) representation 
of $\Sym$ over $\mathbb{C}$ as defined in \cite{Schur1911}*{\S 22} and let 
$X_n^-=sgn\otimes X_n^+$, its associate representation.
\begin{enumerate}
\item\label{sch}	$X_n^\pm$ is an irreducible representation of $\Sym$ of degree $\delta(\Sym)$.
	The representations $X_n^+$ and $X_n^-$ are equivalent only for odd $n$.
\item\label{wal}	For any prime number $p$, the modular reduction $\overline{X_n^\pm}=Y_n^\pm$ 
	over a splitting field of characteristic $p$ 
	is irreducible unless $\pdiv n$ and $n$ is odd. 
	In this case, $\overline{X_n}$ has two distinct composition factors $Y_n^+$ and $Y_n^-$ 
	of degree $\delta(\Sym)$. \\
	Moreover, the associate representations $Y_n^+$ and $Y_n^-$ are non-equivalent 
	exactly if $n$ is even and $\pnotdiv n$, or $n$ is odd and $\pdiv n$. 
\item \label{wal2} Let $n\geq 4$. Assume that $T$ is an absolutely irreducible 
	representation of $\Sym$, and let $t\in\Sym$ be an element of order $3$ 
	projecting to a $3$-cycle in $S_n$.
      If $t^T$ has a quadratic minimal polynomial, then $T$ is equivalent to a 
      basic spin representation of $\Sym$, or $T$ is an ordinary $2$-dimensional 
      representation of $\tilde{S}_4$ which contains $z$ in its kernel. 
\end{enumerate}
\end{theorem}
\begin{proof} For \eqref{sch}, see \cite{Schur1911}*{\S 22 and \S 23.(VII)}, or 
\cite{HoffmanHumphreys1992}*{Theorems 6.2 and 6.8}; for \eqref{wal} and \eqref{wal2} see 
\cite{Wales1978}*{Theorems 7.7 and 8.1}, respectively. The case $n=4$ in \eqref{wal2} is
handled by direct inspection.
\end{proof}

For notational convenience, we identify a (matrix) representation $R$ of $\Sym$ of degree $d$ 
over a field $K$ with the sequence $( t_1^R,\ldots, t_{n-1}^R )\in\left (\GL_d(K)\right )^{n-1}$. 
By $I$ we denote the identity matrix and by $0$ we denote the zero matrix, 
both of suitable degree over $K$; moreover, $Z=z^R=\pm I$. For a given matrix $M$ 
over $K$ and any integer $\mu$, we write $\mu M$ instead of $(\mu\cdot 1_K) M$.\\
In the sense of Theorem \ref{bsp} we refer to the absolutely irreducible representations 
$X_n^\pm$ or $Y_n^\pm$ as the \textit{basic spin representations} of $\Sym$ for $n\geq 4$.
We call any composition factor of a given representation of $\Sym$ which is equivalent to a 
basic spin representation simply a \textit{basic spin composition factor}.

\begin{definition} Let $T=(T_1,\ldots,T_{n-1})$ be a representation of $\Sym$ over a field $K$. 
We say that $T$ satisfies $\Delta$ if 
\begin{equation}\tag{$\Delta$}\label{HYP}
(ZT_iT_{i+1})^2 + ZT_iT_{i+1} + I = 0\quad\mbox{for}\ i\in\{1,\ldots, n-2\}.
\end{equation}
\end{definition}
Let $n\geq 4$ and $T$ be an $\Sym$-representation over a field $K$. 
As $zt_it_{i+1}\in\Sym$ has order $3$ and projects to 
$(i+2,i+1,i)$ in $S_n$, all elements $zt_it_{i+1}$ ($i=1,\ldots,n-2$) are conjugate 
in $\Sym$. 
Hence, if \ref{HYP}($T$) holds for some $i$, it already holds for all $i=1,\ldots,n-2$ and 
every representation
equivalent to $T$ satisfies \ref{HYP} as well; in particular, if $T$ is an ordinary representation
and $p$ is a prime, then its $p$-modular reduction satisfies \ref{HYP}. 
Let $t:=zt_it_{i+1}$ and assume that $t^T$ satisfies \ref{HYP}. 
If $\Char K\neq 3$, this means that $1$ is not an eigenvalue of $t^T$. 
Let $\omega$ be a primitive cube root of unity in some extension 
field of $K$; as $t$ is conjugate to $t^2$, we see that both, $\omega$ and $\omega^2$, 
are eigenvalues of $t^T$. In particular, its minimal polynomial is $x^2+x+1\in K[x]$.
If $\Char K=3$, then we shall suppose that $z\notin\ker T$; thus the kernel of $T$ is trivial 
and the minimal polynomial of $t^T$ is quadratic.

\begin{proposition} \label{prop}
Let $n\geq 4$ and $T=(T_1,\ldots,T_{n-1})$ be a representation of $\Sym$ over a 
splitting field $K$ of characteristic $p\geq 0$. 
For $n=4$ or $p=3$ we shall assume that $z\notin\ker T$. 
If $T$ satisfies \ref{HYP}, then all $T$-composition factors are basic spin.
\end{proposition}
\begin{proof}
If $T$ is reducible, then adapting a basis with respect 
to any bottom composition factor $T_\circ$ of $T$ forces \ref{HYP} to be true for both, 
$T_\circ$ and the corresponding factor representation. By iteration and part \eqref{wal2} of 
Theorem \ref{bsp}, every composition factor of $T$ is basic spin. 
\end{proof}

\section*{Doubled Degree}

\begin{definition}\label{C0} If $T=(T_1,\ldots,T_{n-2})$ is a basic spin representation
of $\sym$ then define the sequence 
$\fS(T)=(T^\uparrow_1,\ldots,T^\uparrow_{n-1})$ by 
\[	T^\uparrow_i = \begin{bmatrix} T_i & \\ & -T_i\end{bmatrix}\
	(1\leq i\leq n-3),\ 
	T^\uparrow_{n-2} = \begin{bmatrix} T_{n-2} & -I \\ & -T_{n-2}\end{bmatrix},
	\quad T^\uparrow_{n-1} = \begin{bmatrix} & I\\ -I& \end{bmatrix}.
\]
\end{definition}

\begin{lemma} Let $T$ be a basic spin representation of $\sym$ over a field $K$ of 
characteristic $p\geq 0$.
Then $\fS(T)$ is a basic spin representation of $\Sym$ over $K$ if either $n$ is odd and 
$\pnotdiv n$, or $n$ is even and $\pdiv (n-1)$. Otherwise, $\fS(T)$ has exactly two 
basic spin composition factors.
\end{lemma}
\begin{proof} Straightforward calculations show that $\fS(T)$ 
satisfies the defining relations of $\Sym$. Thus $\fS(T)$ is a 
representation of $\Sym$ of degree $2\delta(\sym)$. As \ref{HYP}($\fS(T)$) holds, the claim 
follows from Proposition \ref{prop} and Theorem \ref{bsp}. 
\end{proof}

\section*{Stationary Degree}
We are left to deal with the cases where $\delta(\sym)=\delta(\Sym)$, that is,
either $n$ is odd and divisible by $p$, or $n$ is even and $n-1$ is not divisible by $p$.
Given a basic spin representation $(T_1,\ldots,T_{n-2})$ of $\sym$, 
the basic approach is to find an element $T_{n-1}$ such that 
the defining relations of $\Sym$ and the condition \ref{HYP} for $i=n-2$ are satisfied by 
$(T_1,\ldots,T_{n-2},T_{n-1})$. This yields the following equations
\begin{align}
	T_{n-1}^2 & = -I. \label{cond1}\\
  	T_{n-1}T_i + T_iT_{n-1} & = 0\quad\mbox{for all}\ i=1,\ldots, n-3;  \label{cond2}  \\ 
  	T_{n-1}T_{n-2}+T_{n-2}T_{n-1} & = I.\label{cond3}
\end{align}
Note that $(T_{n-2}T_{n-1})^3=Z$ and $(T_{n-2}T_{n-1})^2+ZT_{n-2}T_{n-1}+I=0$ imply \eqref{cond3}. 
Conversely, 
suppose that $T_{n-1}T_{n-2}+T_{n-2}T_{n-1} = I$ and $T_{n-1}^2=Z$. 
We deduce $(T_{n-2}T_{n-1})^2+ZT_{n-2}T_{n-1}+I=0$ and hence $(T_{n-2}T_{n-1})^3=Z$.
Thus, if the equations 
\eqref{cond1}--\eqref{cond3} hold, then by Proposition \ref{prop}, 
$(T_1,\ldots,T_{n-2},T_{n-1})$ is a basic spin representation of $\Sym$.\\

Recall that for an integer $\mu$ and a matrix $M$ defined over the field $K$ we read 
$\mu M$ as $(\mu\cdot 1_K)M$. 

\begin{lemma}\label{complemma}
Let $T=(T_1,\ldots,T_{n-2})$ be a representation of $\sym$ satisfying \ref{HYP}, 
and consider $J=\sum_{k=1}^{n-2}kT_k$. Then the following equations hold:
\begin{align*} 
J^2 & =\left (-(n-2)^2+\sum_{k=1}^{n-3}k\right )I;\tag{1}\label{comp1}\\
JT_i+T_iJ & = 0\quad\mbox{for all}\ i=1,\ldots, n-3;\tag{2}\label{comp2}\\
JT_{n-2}+T_{n-2}J & = (1-n)I.\tag{3}\label{comp3}
\end{align*}
\end{lemma}
\begin{proof}
Using the relation $(T_iT_{i+1})^3=Z$ and \ref{HYP} we get 
$T_{i+1}T_i+T_iT_{i+1}=I$ for all $i=1,\ldots,n-3$;
moreover, we have $T_kT_i+T_iT_k=0$ for all $i,k$ with $|i-k|>1$.
Now we compute
 	 \begin{equation*}
 	 \begin{split}
 	 J^2 & = \sum_{k=1}^{n-2}k^2T_k^2+\sum_{k=1}^{n-3}k(k+1)(T_kT_{k+1}+T_{k+1}T_k)\\
  		& = \left (-\sum_{k=1}^{n-2}k^2+\sum_{k=1}^{n-3}k(k+1)\right )I
	 	= \left (-(n-2)^2+\sum_{k=1}^{n-3}k\right )I, 
 	 \end{split} 
 	 \end{equation*}
and similarly, for $i=1,\ldots,n-3$,
\[JT_i+T_iJ = 2i T_i^2+(i-1)(T_{i-1}T_i+T_iT_{i-1})+
 	 (i+1)(T_{i+1}T_i+T_iT_{i+1})=0,\]
and $JT_{n-2}+T_{n-2}J= -2(n-2)I+(n-3)I = (1-n)I$.
\end{proof}

\subsection*{Case I} If $n$ is divisible by $p$, then we may extend our given 
representation of $\sym$ to a representation of $\Sym$ as follows:
\begin{lemma}\label{C1} Let $T=(T_1,\ldots,T_{n-2})$ be a 
basic spin representation of $\sym$ over a field $K$ of characteristic $p$.
If $\pdiv n$, then the equations \eqref{cond1}, \eqref{cond2} and \eqref{cond3} 
are satisfied by 
\[ T_{n-1} = 
\left\{ \begin{array}{ll}
	J &\mbox{if } p> 2\\
		J & \mbox{if } p=2,\ n\equiv 0\Mod 4\\
		J+I & \mbox{if } p=2,\ n\equiv 2\Mod 4\\
\end{array}\right\}\mbox{ where }\ J=\sum_{k=1}^{n-2} k T_k.
		\]
In particular, $(T_1,\ldots,T_{n-2},T_{n-1})$ is a basic spin representation of $\Sym$ over $K$.
\end{lemma}

\begin{proof} By part \eqref{comp1} of Lemma \ref{complemma},
	$J^2=\left (-(n-2)^2+\sum_{k=1}^{n-3}k\right )I$.
 	 If $n$ is odd, or $n$ is even and $p\neq 2$, then from $\pdiv n$ 
 	 we deduce $p\geq 3$ and verify \eqref{cond1} easily. If $p=2$ and $n\equiv 0\Mod 4$, then
 	 $\sum_{k=1}^{n-3}k=1$; if $p=2$ and $n\equiv 2\Mod 4$, then 
 	 $T_{n-1}^2=J^2+I$ and $\sum_{k=1}^{n-3}k=0$. Hence \eqref{cond1} holds also for $p=2$.
 	 The conditions \eqref{cond2} and \eqref{cond3} follow directly from
 	 \eqref{comp2} and \eqref{comp3} of Lemma \ref{complemma}, respectively.
\end{proof} 	
 
\begin{definition} We denote the construction described in Lemma \ref{C1} by 
$\fSone$. 
\end{definition}

\begin{remark} 
 Let  $T=(T_1,\ldots,T_{n-2})$ be a basic spin representation 
 of $\sym$ for some odd $n\geq 7$ over a field $K$ of characteristic $p$ such that $\pdiv n$.
 Consider the $K\Sym$-module $V$ corresponding to $\fS(T)$, and let 
 \[ M=\begin{bmatrix} M_1&M_2\\M_3&M_4 \end{bmatrix}\in\End_{K\Sym}(V)\]
 where $M_1$, $M_2$, 
 $M_3$ and $M_4$ are block matrices of equal size over $K$. From $MT^\uparrow_i=T^\uparrow_iM$ 
 for all $i=1,\ldots,n-1$ we instantly get $M_1=M_4$ and $M_2=-M_3$, as well as
 \begin{align}
	 M_1T_i & = T_iM_1 \quad\mbox{for all}\ i=1,\ldots, n-3;\label{cond4}\\
  	 M_1T_{n-2} & = T_{n-2}M_1+M_2.\label{cond5}
\end{align}
Since $\delta(\sym)=\delta(\tilde{S}_{n-2})$, we see that $(T_1,\ldots,T_{n-3})$ 
is already a basic spin representation of $\tilde{S}_{n-2}$.
Hence, by \eqref{cond4} and Schur's Lemma, $M_1$ is scalar; 
thus \eqref{cond5} yields $M_2=0$. This shows 
$\End_{K\Sym}(V) \cong K$, so $V$ is indecomposable. 
\end{remark}

\subsection*{Case II}
Now suppose that $n=2(k+1)$ and $\pnotdiv n(n-1)(n-2)$, so 
$\delta(\Sym)=\delta(\sym)$. As our approach is inductive, we may assume that we are given 
a basic spin representation $T$ of $\sym$ which itself has been constructed from a 
basic spin representation $S$ of $\tilde{S}_{n-2}$, that is, 
$T=\fS(S)$. 

\begin{lemma}\label{C2}  Provided $n\geq 6$ is even, let $T=(T_1,\ldots,T_{n-3})$ be a 
basic spin representation of $\tilde{S}_{n-2}$ over a field $K$ of characteristic 
$p\geq 0$ such that $\pnotdiv n(n-1)(n-2)$. Then 
\[ T_{n-1} = \begin{bmatrix} 
		-\alpha J & (\beta-1)I\\
		\beta I & \alpha J\end{bmatrix}\]
	where
\[ J = \sum_{k=1}^{n-3} kT_k,\ 
		\alpha =  (n-1)^{-1}\left ( 1\pm\sqrt{-n(n-2)^{-1}}\right ) \mbox{ and }
		\beta =  (n-2)\alpha\]
extends $\fS(T)$ to a basic spin representation of $\Sym$ over $K(\alpha)$.
\end{lemma}

\begin{proof} 
We have to verify the relations \eqref{cond1}, \eqref{cond2} and \eqref{cond3}.
The calculations depend on Lemma \ref{complemma}. From \ref{complemma}.\eqref{comp1} we get 
$J^2 = -1/2(n-3)(n-2)I$ and hence $T_{n-1}^2=-I$ as
$\alpha^2J^2+(\beta^2-\beta+1)I = 0$. 
From \ref{complemma}.\eqref{comp2} and \ref{complemma}.\eqref{comp3} we have
$JT_i+T_iJ=0$ for all $i=1,\ldots, n-4$
and $JT_{n-3}+T_{n-3}J=-(n-2)I$, respectively, thus we deduce
$T_{n-1}T_i^\uparrow+T_i^\uparrow T_{n-1}=0$ for $i=1,\ldots,n-3$ and
$T_{n-1}T_{n-2}^\uparrow+T_{n-2}^\uparrow T_{n-1}=I$.
\end{proof}

\subsection*{Case III} Let $n=2k+2$ such that $\pdiv (n-2)$ and $\pnotdiv n$, and suppose that 
we are given a basic spin representation $T=(T_1,\ldots,T_{n-3})$ 
of $\tilde{S}_{n-2}$ over a field $K$ of characteristic $p>2$.
In this situation the approach of 
Lemma \ref{C2} fails since the matrix $J$ squares to zero. 
But as in Case II where $T_{n-1}$ is subdivided into four blocks of equal size, we again 
divide the upper left corner into blocks. 
Since $\delta(\tilde{S}_{n-2})=\delta(\tilde{S}_{n-3})=2^{k-1}$ and $\pnotdiv (n-3)$,
we may suppose that $T$ has been constructed by means of $\fSone$, thus
$T_{n-3}=\sum_{k=1}^{n-4}kT_k$. Furthermore, we may also assume that the 
representation $(T_1,\ldots,T_{n-4})$ of $\tilde{S}_{n-3}$ has been 
constructed by $\fS$ from a basic spin representation of $\tilde{S}_{n-4}$.

\begin{lemma}\label{C3} 
Let $n\geq 8$ be even. If $T=(T_1,\ldots,T_{n-5})$ is a basic spin representation of 
$\tilde{S}_{n-4}$ over a field $K$ of characteristic $p\geq 3$ such that
$\pdiv (n-2)$, define
\[ T_{n-1} = \begin{bmatrix} 
		J & -I\\
		& -J\end{bmatrix}\quad\mbox{where}\quad 
J=\pm\sqrt{-1}\begin{bmatrix} \sum_{k=1}^{n-5}kT_k & 2I\\
			   2I & -\sum_{k=1}^{n-5}kT_k\end{bmatrix}.\]
The matrix $T_{n-1}$ extends the basic spin representation $\fS\fSone\fS(T)$ 
of $\sym$ to $\Sym$. The field of definition is $K(\sqrt{-1})$.
\end{lemma}

\begin{proof} Let $J_\circ = \sum_{k=1}^{n-5}kT_k$; 
by \ref{complemma}.\eqref{comp1}, 
$ J_\circ^2=( -(n-5)^2+\sum_{k=1}^{n-6}k )I=-3I$ as $n-2$ is divisible by $p$,
and thus $T_{n-1}^2=-I$, so \eqref{cond1} holds. For the sake of clarity, 
we summarize the representations at hand:
\begin{itemize} 
\item $\tilde{S}_{n-3}$:\quad $\fS(T)=(T_1^\uparrow,\ldots,T_{n-4}^\uparrow)$,
\item $\tilde{S}_{n-2}$:\quad $\fSone\fS(T)=(T_1^\uparrow,\ldots,T_{n-4}^\uparrow,T_{n-3})$ with 
$T_{n-3}=\sum_{k=1}^{n-4}kT_k^\uparrow$,
\item $\tilde{S}_{n-1}$:\quad 
$\fS\fSone\fS(T)=(T_1^{\arr},\ldots,T_{n-4}^{\arr},T_{n-3}^\uparrow, T_{n-2}^\uparrow)$.
\end{itemize}
For \eqref{cond2}, we need to show that  
\[T_{n-1}T_i^{\arr}+T_i^{\arr}T_{n-1}=0\quad\mbox{for } i=1,\ldots,n-4\quad\mbox{and}\quad
T_{n-1}T_{n-3}^{\uparrow}+T_{n-3}^{\uparrow}T_{n-1}=0.\]
The left hand sides are
\[\begin{bmatrix} J & -I \\ & -J\end{bmatrix}\begin{bmatrix}T_i^\uparrow& \\ & -T_i^\uparrow\end{bmatrix}
+\begin{bmatrix}T_i^\uparrow& \\ & -T_i^\uparrow\end{bmatrix}\begin{bmatrix} J & -I \\ & -J\end{bmatrix}
=\begin{bmatrix} JT_i^\uparrow+T_i^\uparrow J &  \\ & JT_i^\uparrow+T_i^\uparrow J\end{bmatrix}\]
and
\[\begin{bmatrix} J & -I \\ & -J\end{bmatrix}\begin{bmatrix}T_{n-3}&-I \\ & -T_{n-3}\end{bmatrix}
+\begin{bmatrix}T_{n-3}& -I\\ & -T_{n-3}\end{bmatrix}\begin{bmatrix} J & -I \\ & -J\end{bmatrix}
=\begin{bmatrix} X &\\ & X\end{bmatrix}\] with $X=JT_{n-3}+T_{n-3} J$.
By \eqref{comp2} and \eqref{comp3} of Lemma \ref{complemma}, $J_\circ T_i+T_iJ_\circ=0$ for 
$i=1,\ldots,n-6$, and
$J_\circ T_{n-5}+T_{n-5}J_\circ =(4-n)=2I$, respectively. Thus
\begin{align*}
&JT_i^\uparrow+T_i^\uparrow J\\
& =\begin{bmatrix} J_\circ & 2I \\ 2I& -J_\circ\end{bmatrix}\begin{bmatrix}T_i& \\ & -T_i\end{bmatrix}
+\begin{bmatrix}T_i& \\ & -T_i\end{bmatrix}\begin{bmatrix} J_\circ & 2I \\ 2I& -J_\circ\end{bmatrix}
=0\ (1\leq i\leq n-6),\\
&JT_{n-5}^\uparrow+T_{n-5}^\uparrow J\\
& =\begin{bmatrix} J_\circ & 2I \\ 2I& -J_\circ\end{bmatrix}\begin{bmatrix}T_{n-5}& -I\\ & -T_{n-5}\end{bmatrix}
+\begin{bmatrix}T_{n-5}& -I\\ & -T_{n-5}\end{bmatrix}\begin{bmatrix} J_\circ & 2I \\ 2I& -J_\circ\end{bmatrix}
=0,\\
&JT_{n-4}^\uparrow+T_{n-4}^\uparrow J=\begin{bmatrix} J_\circ & 2I \\ 2I& -J_\circ\end{bmatrix}\begin{bmatrix}& I\\-I & \end{bmatrix}
+\begin{bmatrix}& I\\ -I& \end{bmatrix}\begin{bmatrix} J_\circ & 2I \\ 2I& -J_\circ\end{bmatrix}
=0,\quad{and}\\
&J T_{n-3}+T_{n-3} J = \sum_{k=1}^{n-4} k\left (T_k^\uparrow J+J T_k^\uparrow \right )=0.
\end{align*}
It remains to verify \eqref{cond3} which is checked easily.
\end{proof}

\begin{remark} Let $n\geq 6$ be even and assume that $R$ 
is a basic spin representation of 
$\tilde{S}_{n-2}$. In Lemma \ref{C2} we have given two different 
matrices $T_{n-1}^+$ and $T_{n-1}^-$ which extend $T=\fS(R)=(T_1,\ldots,T_{n-2})$ 
to basic spin 
representations $T^+$ and $T^-$ of $\Sym$, respectively. 
Then $T^+$ and $T^-$ are non-equivalent as any transformation matrix has to commute with 
$T_i$ for $i=1,\ldots,n-2$, and hence is scalar
by Schur's Lemma. Moreover, if $V$ denotes the module afforded by $\fS(T)$, then 
$\begin{bmatrix} I& -T_{n-1}^+\end{bmatrix}$ and 
$\begin{bmatrix} I& -T_{n-1}^-\end{bmatrix}$ are bases of the submodules of $V$
corresponding to $T^+$ and $T^-$, respectively; in particular, $V$ is decomposable. 
A similar remark applies to the construction of Lemma \ref{C3}.
\end{remark}

Finally, it remains to give a basic spin representation for $\tilde{S}_4$.
But we can simply start with
a suitable representation $T$ of $\tilde{S}_2$ and then apply the construction. 
We refer to the constructions of Lemma \ref{C2} and Lemma \ref{C3} as 
$\fStwo$ and $\fSthree$, respectively.
For example, 
\begin{itemize} 
\item over $\GF(2)$ we get
	$\fS(T)=\left (\begin{bmatrix} 1&1\\&1\end{bmatrix},
		     	 \begin{bmatrix} &1\\1&\end{bmatrix}\right )$ for $\tilde{S}_3$, so
    	\[\fSone\fS(T)=\left (\begin{bmatrix} 1&1\\ &1\end{bmatrix},
		     	 \begin{bmatrix} &1\\1&\end{bmatrix},
		     	 \begin{bmatrix} 1&1\\&1\end{bmatrix}\right )\]
 	is a basic spin representation of $\tilde{S}_4$.
\item over $\GF(9)$ with primitive element $\zeta$, the construction $\fSone$ yields
	$\left (\begin{bmatrix} \zeta^2\end{bmatrix},
	\begin{bmatrix} \zeta^2\end{bmatrix}\right )$ as a representation of $\tilde{S}_3$;
	hence
    	\[\fS\fSone(T)=\left (\begin{bmatrix} \zeta^2 & \\ & \zeta^6 \end{bmatrix},
		     	 \begin{bmatrix} \zeta^2 & \zeta^4 \\ & \zeta^6\end{bmatrix},
		     	 \begin{bmatrix} & 1\\ \zeta^4 & \end{bmatrix}\right )\]
 	is a basic spin representation of $\tilde{S}_4$.
\item over $\GF(p^2)$ for $p>3$ with primitive element $\zeta$ and  
$\omega=\zeta^{(p^2-1)/4}$, 
    	\[\fStwo(T)=\left (\begin{bmatrix} \omega&-1\\&-\omega \end{bmatrix},
		     	 \begin{bmatrix} & 1 \\ -1 & \end{bmatrix},
		     	 \begin{bmatrix} -\alpha\omega & 2\alpha-1 \\ 
		     	 2\alpha & \alpha\omega\end{bmatrix} \right )\]
		     	 with $\alpha=3^{-1}(1+\omega\sqrt{2})\in\GF(p^2)$,
 	a basic spin representation of $\tilde{S}_4$. 
\end{itemize}

\begin{corollary} Let $n\geq 4$. 
For $K=\mathbb{C}$, $\GF(2)$, or $\GF(p^2)$ with $p\geq 3$, iterative use of
the constructions $\fS$, $\fSone$, $\fStwo$, or $\fSthree$,  
as described in Table \ref{table}, provides a basic spin representation of 
$\Sym$ over $K$.
\end{corollary}

\begin{table}[h]\caption{Basic Spin Constructions over $K$}\label{table}
\begin{center}
\begin{tabular}{llll}
 \toprule
 \addlinespace[6pt]
 $K$	& 	& $n$ odd 	& $n$ even\\
 \addlinespace[4pt]
 \midrule
 \addlinespace[6pt]
$\mathbb{C}$ &	& $\fS$	& $\fStwo$ \\ 
\addlinespace[3pt]
\midrule
\addlinespace[6pt]
$\GF(2)$ &	& $\fS$	& $\fSone$ \\ 
\addlinespace[3pt]
\midrule
\addlinespace[6pt]
 $\GF(p^2)$		&	$\pnotdiv n(n-1)(n-2)$ 	& $\fS$ 	&  $\fStwo$\\ 
 \addlinespace[6pt]
 $(p\geq 3$) 	&	$\pdiv n $ 			& $\fSone$	& $\fSone$ \\ 
 \addlinespace[6pt]
 	 		&	$\pdiv (n-1)$ 		& $\fS$	& $\fS$ \\ 
 \addlinespace[6pt]
 	  		&	$\pdiv (n-2)$ 		& $\fS$	& $\fSthree$ \\ 
 \addlinespace[6pt]
 \bottomrule
\end{tabular}
\end{center}
\end{table}

\begin{ack} The author is grateful to Wolfgang Lempken, Klaus Lux and J\"urgen M\"uller 
for their suggestions and useful hints.
\end{ack}

\bibliographystyle{plain}

\end{document}